\documentclass[12pt]{article}
\usepackage{amsmath}
\usepackage{amssymb}
\usepackage{amsthm}
\usepackage[utf8]{inputenc}
\usepackage[english]{babel}
\usepackage{graphicx}
\usepackage{rotating}
\usepackage{array}
\newtheorem{theorem}{Theorem}[section]

\newtheorem{lemma}[theorem]{Lemma}
\newcommand\pfree{\mathop{\mbox{$(q^m-1)$-$\mathit{free}$}}}
\newcommand\ufree{\mathop{\mbox{$u$-$\mathit{free}$}}}

\newcommand\llfree{\mathop{\mbox{$l_1$-$\mathit{free}$}}}
\newcommand\lllfree{\mathop{\mbox{$l_2$-$\mathit{free}$}}}

\begin{document}
\title{Existence of Primitive Pairs with Prescribed Traces over Finite Fields }
\author{Hariom Sharma, R. K. Sharma}

\date{}
\maketitle
\begin{center}
	\textit{Department of Mathematics\\
	Indian Institute of Technology Delhi\\
	New Delhi, 110016, India}
\end{center}

\begin{abstract}
Let $F=\mathbb{F}_{q^m}$, $m>6$, $n$ a positive integer, and $f=p/q$ with $p$, $q$ co-prime irreducible polynomials in $F[x]$ and deg$(p)$ $+$ deg$(q)= n$. A sufficient condition has been obtained for the existence of primitive pairs $(\alpha, f(\alpha))$ in $F$ such that for any prescribed $a, b$ in $E=\mathbb{F}_q$, Tr$F/E (\alpha) = a$ and Tr$F/E (\alpha^{-1}) = b$. Further, for every positive integer $n$, such a pair definitely exists for large enough $(q,m)$. The case $n = 2$ is dealt separately and proved that such a pair exists for all $(q,m)$ apart from at most $64$ choices.
\end{abstract}

\textbf{Keywords:} Finite Fields, Characters, Primitive element\\
2010 Math. Sub. Classification: 12E20, 11T23
\footnote{email: hariomsharma638@gmail.com, rksharmaiitd@gmail.com}
\section{Introduction}

Let $\mathbb{F}_q$ be a finite field with $q=p^k$ elements where $p,k\in \mathbb{N}$ and $p$ is a prime number. Multiplicative cyclic group of non zero elements is denoted by $\mathbb{F}_q^*$ and the generator of $\mathbb{F}_q^*$ is called \textit{primitive element} in $\mathbb{F}_q$. For a positive integer $m(\geq 7)$, let $\mathbb{F}_{q^m}$ be the field extension of $\mathbb{F}_q$ of degree $m$. Indeed, an element $\alpha \in \mathbb{F}_{q^m}$ is primitive if and only if it is a zero of an irreducible polynomial of degree $m$ over $\mathbb{F}_q$, the irreducible polynomial is known as \textit{primitive polynomial}. Also, the trace of $\alpha$ over $\mathbb{F}_q$, denoted by $\text{Tr}_{\mathbb{F}_{q^m}/ \mathbb{F}_q}(\alpha)$ is given by $\alpha+\alpha^q+\cdots+\alpha^{q^{m-1}}$.

 Primitive elements are used as fundamental tool in many cryptographic schemes(e.g., Diffie-Hellmen key exchange protocol). More precisely, they find several applications in cryptography and coding theory \cite{PPR3}. Therefore, the study of primitive elements and primitive polynomials is an active area of research. Another interesting problem related to primitive elements is study of primitive pairs. For a rational function $f(x)\in \mathbb{F}_{q^m}(x)$ and $\alpha \in \mathbb{F}_{q^m}$, we call  a pair $(\alpha, f(\alpha))\in \mathbb{F}_{q^m} \times \mathbb{F}_{q^m}$  a \textit{primitive pair} in $\mathbb{F}_{q^m}$ if both $\alpha$ and $f(\alpha)$ are primitive elements in $\mathbb{F}_{q^m}$. Generally, $f(\alpha)$ not necessarily be primitive element for a primitive $\alpha\in \mathbb{F}_q$, for example, if $f(x)=x^2+2\in \mathbb{F}_5(x)$, then $f(\alpha)$ is not primitive for any primitive $\alpha\in \mathbb{F}_5$.
 
 It is both of theoretical importance and natural challenge to establish the existence of primitive elements with some prescribed conditions. Many researchers worked in this direction \cite{PPR5,PPR6,cao2014primitive}. D. Jungnickel and S. A. Vanstone \cite{jungnickel1989primitive} proved the existence of primitive element $\omega$ in $\mathbb{F}_{q^m}$ with prescribed $\text{Tr}_{\mathbb{F}_{q^m}/ \mathbb{F}_q}(\omega)$ in $\mathbb{F}_q$ for all pairs $(q,m)$ excluding finitely many. Cohen \cite{cohen2005primitive}, extended the result and established it for each pair except $\text{Tr}_{\mathbb{F}_{q^m}/ \mathbb{F}_q}(\omega)= 0$ if $m=2$ and $(4,3)$. Chou and Cohen \cite{chou2001primitive} resolved completely the question of whether there exists a primitive element $\alpha$ in $\mathbb{F}_{q^m}$ such that both $\alpha$ and $\alpha^{-1}$ have trace zero over $\mathbb{F}_q$.
 
 First in $1985$, Cohen \cite{PPR1} studied the existence of primitive pairs $(\alpha, f(\alpha))$ over finite field $\mathbb{F}_q$ where $f(x)=x+a, a\in \mathbb{F}_q$. In 2014, Cao and Wang \cite{cao2014primitive} considered the existence
 of the primitive pairs  with $f(x)=\frac{x^2+1}{x}$ in the finite field $\mathbb{F}_{q^m}$ and got that when $m \geq 29$,
 there are such elements with 
$\text{Tr}_{\mathbb{F}_{q^m}/ \mathbb{F}_q}(\alpha)=a$ and $\text{Tr}_{\mathbb{F}_{q^m}/ \mathbb{F}_q}(\alpha^{-1})=b$ for any pair of prescribed $a, b\in \mathbb{F}_q^*$. For the same rational function, in 2018, Anju, Sharma and Cohen \cite{PPR5}, obtained a sufficient condition for  existence of a primitive pair with $\text{Tr}_{\mathbb{F}_{q^m}/ \mathbb{F}_q}(\alpha)=a$ for any prescribed $a\in \mathbb{F}_q$ and proved the existence of such elements for all pairs $(q,m)$, $m\geq 5$. Later in 2019, Sharma and Gupta \cite{PPR6} generalized the rational function to $\lambda_A(x)$, where  $$\lambda_A(x)=\frac{ax^2+bx+c}{dx+e},$$ 
for any matrix $A=\begin{pmatrix}
a&b&c\\0&d&e
\end{pmatrix}\in M_{2\times 3}(\mathbb{F}_{q^m})$ of rank $2$ and if $\lambda_A(x)=\beta x$ or $\beta x^2$ for some $\beta \in \mathbb{F}_{q^n}$ then $\beta =1$. Next, for $m\geq 7$, they proved that for any such matrix $A$ over finite field $\mathbb{F}_{q^m}$ of  characteristic  $2$ there exists primitive pairs $(\alpha, \lambda_A(\alpha))$ in $\mathbb{F}_{q^m}$ such that for any prescribed $\mu, \nu\in \mathbb{F}_q$, $\text{Tr}_{\mathbb{F}_{q^m}/ \mathbb{F}_q}(\alpha)=\mu$ and $\text{Tr}_{\mathbb{F}_{q^m}/ \mathbb{F}_q}(\alpha^{-1})=\nu$ except for at most $25$ choices of $(q,m)$.    

 In this paper, we take $f(x)$ to be  more general rational function and $\mathbb{F}_q$ be finite field of any prime characteristic $p$.  We propose the problem as follows.
  For a rational function  $f(x)=\frac{a_{n_1}x^{n_1}+ \cdots+a_0}{b_{n_2}x^{n_2}+\cdots+b_0} \in \mathbb{F}_{q^m}(x)$, we assume that  $p(x)=a_{n_1}x^{n_1}+ \cdots+a_0$, $q(x)=b_{n_2}x^{n_2}+\cdots+b_0$ and $a_{n_1}, b_{n_2}\neq 0$. For $n_1, n_2\in \mathbb{N}\cup \{0\}$ , define a subset of $\mathbb{F}_{q^m}(x)$ by
\[
  R_{n_1, n_2} =\left\lbrace f(x) \in \mathbb{F}_{q^m}(x) \;\middle|\;
  \begin{tabular}{@{}l@{}}
     $p(x)$ and $q(x)$ are co-prime irreducible polynomials\\ over $\mathbb{F}_{q^m}$ with deg$(p(x)=n_1$, deg$(q(x)=n_2))$ 
   \end{tabular}
  \right\rbrace,
\] and set of pairs as
\[
Q_{n_1,n_2} =\left\lbrace (q,m)  \;\middle|\;
\begin{tabular}{@{}l@{}}
for each $f(x)\in R_{n_1,n_2}$, there exists a primitive pair \\$(\alpha, f(\alpha)) $ in  $ \mathbb{F}_{q^m} $ such that, for any prescribed $a$ and $b$ \\in $\mathbb{F}_q,$ $\text{Tr}_{\mathbb{F}_{q^m}/ \mathbb{F}_q}(\alpha)=a$ and $\text{Tr}_{\mathbb{F}_{q^m}/ \mathbb{F}_q}(\alpha^{-1})=b$\end{tabular}
\right\rbrace.
\]
Let $R_n=\bigcup\limits_{n_1+n_2=n}R_{n_1,n_2}$. 
and $Q_n=\bigcap\limits_{n_1+n_2=n}Q_{n_1,n_2}$.\\
 For each $n\in \mathbb{N}$, first we establish a sufficient condition on $q^m$ such that $(q,m)\in Q_{n}$. Further using a sieving modification of this sufficient condition we proved following result.
 
 \begin{theorem}\label{case n=2}
 Let $m, p, k\in \mathbb{N}$ such that $p$ is a prime and $m\geq 7$. If $q=p^k$, then $(q,m)\in Q_2$ unless one of the following holds :\\
 \begin{enumerate}
 \item $2\leq q \leq 16$ or $q=19, 23, 25, 31, 37, 43, 49, 79$ with $m=7;$
 \item $2\leq q\leq 31$ or $q=37, 41, 43, 47, 83$ with $m=8;$
 \item $2\leq q\leq 7$ or $q=9, 11, 16$ with $m=9;$
 \item $2\leq q\leq 7$ with $m=10;$
 \item $q=2,3$ with $m=11;$
 \item $q=2, 3, 4, 5$ with $m=12;$
 \item $q=2$ with $m=14, 16, 18, 20, 24.$
 \end{enumerate}

 \end{theorem}

\section{Preliminaries}
 In this section, we come out with some definitions and results which we shall need further in this article. If $\mathcal{D}$ denotes the set of divisors of $q^m-1$, then for $u\in \mathcal{D}$, an element $w\in \mathbb{F}_{q}^*$ is called $\ufree$, if $w=v^d$, where $v\in \mathbb{F}_{q^m}$ and $d|u$ implies $d=1$. Note that an element $w\in \mathbb{F}_{q^m}^*$ is $\pfree$ if and only if it is primitive. For more fundamentals on characters, primitive elements and finite fields, we refer the reader to \cite{PPR2}.

As a special case of Lemma 10 of \cite{shuqin2004character}, we have an interesting result as following. 
\begin{lemma}
Let $u\in \mathcal{D}$, $\xi\in \mathbb{F}_{q^m}^*$. We have:

 \[
 \sum_{d|u}\frac{\mu(d)}{\phi(d)}\sum_{\chi_d}\chi_d(\xi)=
 \begin{cases}
 \frac{u}{\phi(u)} &\quad\text{ if } \xi \text{is } \ufree,\\
 0 &\quad\text{otherwise.}
 \end{cases}
 \]
where $\mu(\cdot)$ is the M$\ddot{\text{o}}$bius function and $\phi(\cdot)$ is the Euler function, $\chi_d$ runs through all
the $\phi(d)$ multiplicative characters over $\mathbb{F}_{q^m}^*$ with order $d$.

\end{lemma}

Therefore, for each $u\in \mathcal{D}$, 
 \begin{equation}\label{charfree}
\rho_{u} : \alpha \mapsto \theta(u)\sum_{d|u}\frac{\mu(d)}{\phi(d)}\sum_{\chi_d}\chi_d(\alpha),
 \end{equation} 
 
  gives a characteristic function for the subset of $\ufree$ elements of $\mathbb{F}_{q^m}^*$, where $\theta(u)=\frac{\phi(u)}{u}$. \\
  Also, for each $a\in \mathbb{F}_q$,
 \begin{equation}\label{chartrace}
 \tau_a : \alpha \mapsto  \frac{1}{q} \sum \limits_{\psi\in \hat{\mathbb{F}_q}}\psi(\text{ Tr}_{\mathbb{F}_{q^m}/ \mathbb{F}_q}(\alpha)-a)
 \end{equation} is a characterstic function for the subset of $\mathbb{F}_{q^m}$ consisting elements  with $\text{ Tr}_{\mathbb{F}_{q^m}/ \mathbb{F}_q}(\alpha)=a$.
 We shall need the following results of  D. Wang and L. Fu for our next theorem.
 \begin{lemma}\label{lemma1}$[\cite{PPR4}, Theorem~4.5]$ Let $f(x)\in \mathbb{F}_{{q}^d}(x)$ be a rational function. Write $f(x)= \prod_{j=1}^{k}f_j(x)^{n_j}$,
  where $f_j(x)\in \mathbb{F}_{{q}^d}[x]$ are irreducible polynomials and $n_j$ are non zero integers. Let $\chi$ be a multiplicative character of $\mathbb{F}_{q^d}$. Suppose that the rational function $\prod_{i=0}^{d-1}f(x^{q^i})$ is not of the form $h(x)^{\text{ord}(\chi)}$ in $\mathbb{F}_{q^d}(x),$ where ord$(\chi)$ is the smallest positive integer $r$ such that $\chi^r=1$, then we have $$\big|\sum_{\alpha\in \mathbb{F}_{q},f(\alpha)\neq 0, f(\alpha)\neq \infty}\chi(f(\alpha))\big|\leq (d\sum_{j=1}^k deg(f_j)-1)q^{\frac{1}{2}}.$$
\end{lemma}

\begin{lemma}\label{lemma2}$[\cite{PPR4}, Theorem~4.6]$
	Let $f(x), g(x)\in \mathbb{F}_{q^m}(x)$ be rational functions. Write $f(x)= \prod_{j=1}^{k}f_j(x)^{n_j}$,
	where $f_j(x)\in \mathbb{F}_{{q}^m}[x]$ are irreducible polynomials and $n_j$ are non zero integers. Let $D_1=\sum_{j=1}^{k}deg(f_j)$, let $D_2= max(deg(g),0)$, let $D_3$ be the degree of denominator of $g(x)$, and let $D_4$ be the sum of degrees of those irreducible polynomials dividing denominator of $g$ but distinct from $f_j(x)(j=1, 2, \cdots k)$. Let $\chi$ be a multiplicative character of $\mathbb{F}_{q^m}$, and let $\psi$ be a non trivial additive character of $\mathbb{F}_{q^m}$. Suppose $g(x)$ is not of the  form $r(x)^{q^m}-r(x)$ in $\mathbb{F}_{q^m}(x)$. Then we have the estimate 
	 $$\big|\sum_{\alpha\in \mathbb{F}_{q^m},f(\alpha)\neq 0, \infty g(\alpha)\neq \infty}\chi(f(\alpha))\psi(g(\alpha)) \big|\leq (D_1+D_2+D_3+D_4-1)q^{\frac{m}{2}}.$$
\end{lemma}
\section{Sufficient condition}

 For each divisor $l_1, l_2$ of $q^m-1$, $f(x)\in R_n$ and prescribed elements $a, b$ of $\mathbb{F}_q$,  suppose $N_{f, n, a, b}(l_1, l_2)$ denotes the number of elements $\alpha \in \mathbb{F}_{q^m}$ such that $\alpha $ is  $\llfree$,  $f(\alpha)$ is $\lllfree$, $\text{Tr}_{\mathbb{F}_{q^m}/ \mathbb{F}_q}(\alpha)=a$ and $\text{Tr}_{\mathbb{F}_{q^m}/ \mathbb{F}_q}(\alpha^{-1})=b$.\\\\
 We now prove our sufficient condition as follows.

\begin{theorem}\label{Theorem1}
Let $m, n \text{ and }q\in \mathbb{N}$ such that q is a prime power. Suppose 
\begin{align}\label{Main condition}
q^{\frac{m}{2}-2}>(n+2)W(q^m-1)^2.
\end{align} Then $(q,m)\in Q_n$. 
\end{theorem}

\begin{proof}
 To prove the result, we need to show that $N_{f,n, a, b}(q^m-1, q^m-1)>0$ for every $f(x)\in R_n$ and prescribed $a, b\in \mathbb{F}_q$. Let $f(x)\in R_n$ be any rational function and $a, b \in \mathbb{F}_q$. Let $S_1$ be the set of zeros and poles of $f(x)$ in  $\mathbb{F}_{q^m}$ and $S=S_1\cup \{0\}$. Assume $l_1, l_2$ be divisors of $q^m-1$. Then by definition we have $$N_{f,n, a, b}(l_1,l_2)=\sum_{\alpha\in \mathbb{F}_{q^m} \setminus S} \rho_{l_1}(\alpha) \rho_{l_2}(f(\alpha))\tau_a(\alpha)\tau_b(\alpha^{-1})$$ now using (\ref{charfree}) and (\ref{chartrace}), we have 
 \begin{align}\label{key}
 N_{f,n, a, b}(l_1,l_2)=\frac{\theta(l_1)\theta(l_2)}{q^2}\sum_{d_1|l_1, ~d_2|l_2}\frac{\mu(d_1)}{\phi(d_1)} \frac{\mu(d_2)}{\phi(d_2)} \sum_{\chi_{d_1},~\chi_{d_2}} \chi_{f,a,b}(\chi_{d_1}, \chi_{d_2}) \end{align}
 where
 $$ \chi_{f,a,b}(\chi_{d_1}, \chi_{d_2})=\sum_{u, v\in \mathbb{F}_q}\psi_0 (-au-bv)\sum_{\alpha\in \mathbb{F}_{q^m}\setminus S}\chi_{d_1}(\alpha)\chi_{d_2}(f(\alpha))\hat{\psi_0}(u\alpha+v\alpha^{-1}).$$\\
From [Example 5.1, \cite{PPR2}], it follows that, for any given divisors $d_1,d_2$ of $q^m-1$ there exist integers $m_1, m_2$ with $0\leq m_1, m_2<q^m-1$ such that $\chi_{d_1}(x)=\chi_{q-1}(x^{m_1})$ and $\chi_{d_2}(x)=\chi_{q-1}(x^{m_2})$. Thus
$$\chi_{f,a,b}(\chi_{d_1}, \chi_{d_2})=\sum_{u, v\in \mathbb{F}_q}\psi_0 (-au-bv)\sum_{\alpha\in \mathbb{F}_{q^m}\setminus S}\chi_{q^m-1}(\alpha^{m_1} f(\alpha)^{m_2})\hat{\psi_0}(u\alpha+v\alpha^{-1})$$
$$~~~~~~~~~=\sum_{u, v\in \mathbb{F}_q}\psi_0 (-au-bv)\sum_{\alpha\in \mathbb{F}_{q^m}\setminus S}\chi_{q^m-1}(F(\alpha))\hat{\psi_0}(G(\alpha)).$$

where, $F(x)=x^{m_1}f(x)^{m_2}\in \mathbb{F}_{q^m}(x)$ and $G(x)=ux+vx^{-1}\in \mathbb{F}_{q}(x)$.\\
If $G(x)\neq h(x)^{q^m}-h(x)$ for any $h(x)\in \mathbb{F}_{q^m}(x)$ then by Lemma \ref{lemma2} \begin{align}\label{bound}
|\chi_{f,a,b}(\chi_{d_1}, \chi_{d_2})|\leq (n+2)q^{\frac{m}{2}+2}.
\end{align} If $G(x)= h(x)^{q^m}-h(x)$ for some $h(x)\in \mathbb{F}_{q}(x) $ then following [9], it is only possible if $u=v=0$. Hence, if $F(x)\neq h(x)^{q^m-1}$ for any $h(x)\in \mathbb{F}_{q^m}(x)$ by Lemma \ref{lemma1},
\begin{align}\label{support}
|\chi_{f,a,b}(\chi_{d_1}, \chi_{d_2})|\leq nq^{\frac{m}{2}+2}.
\end{align}
 Now,  let us consider the case when  $F(x)= g(x)^{q^m-1} \text{ for some } g(x)\in \mathbb{F}_{q^m}(x) $. If $g(x)=\frac{g_1(x)}{g_2(x)}$ for $g_1(x), g_2(x)\in \mathbb{F}_{q^m}[x]$ with gcd$(g_1(x),g_2(x))=1$, then $x^{m_1}\big(\frac{p(x)}{q(x)}\big)^{m_2}=\big(\frac{g_1(x)}{g_2(x)}\big)^{q^m-1}$, that is \begin{align}\label{initial}
x^{m_1}p(x)^{m_2}g_2(x)^{q^m-1}=g_1(x)^{q^m-1}q(x)^{m_2}.
 \end{align}
 We claim that (\ref{initial}) is possible only if $m_1=m_2=0$. For this, first
we prove that if $m_1$ is $0$, then $m_2$ must also be $0$.  Suppose $m_1=0$, then equation (\ref{initial}) becomes $$p(x)^{m_2}g_2(x)^{q^m-1}=g_1(x)^{q^m-1}q(x)^{m_2}.$$ Let if possible, $m_2\neq 0$,  then $p(x)$ and $q(x)$ being co-prime gives $p(x)$ divides $g_1(x)$, which further gives $g_2(x)^{q^m-1}={g_1}'(x)^{q^m-1}q(x)^{m_2}p(x)^{q^m-m_2-1}$, where ${g_1}'(x)=g_1(x)/p(x)$. $q^m-m_1-1>0$ tells that $p(x)$ divides $g_2(x)$. A contradiction. Hence, $m_1=0 \text{ implies } m_2=0$.\\
  Next if possible, let $m_1\neq 0$. Then from (\ref{initial}), either $x|g_1(x) \text{ or } x|q(x)$. First suppose $x \text{ divides }g_1(x)$. We can restate equation (\ref{initial}) as $$p(x)^{m_2}g_2(x)^{q^m-1}=g_1'(x)^{q^m-1}q(x)^{m_2}x^{q^m-m_1-1}$$ where $g_1'(x)=\frac{g_1(x)}{x}$. Here gcd$(g_1(x),g_2(x))=1$ and $q^m-m_1-1>0$ forces that $x|p(x)$ and $m_2\neq 0$. But $p(x)$ is irreducible, and hence $p(x)=ax$ for some $a\in \mathbb{F}_{q^m}^*$. This gives that $cx^{m_2}g_2(x)^{q^m-1}=g_1'(x)^{q^m-1}q(x)^{m_2}x^{q^m-m_1-1}$, where $c=a^{m_2}$. Here we come up with three possibilities as discussed in following cases.\\
  \textbf{Case 1.} $q^m-m_1-1>m_2$. This gives that $x|g_2(x)$, which is not so.\\
  \textbf{Case 2.} $q^m-m_1-1<m_2$. As $x$ can not divide $q(x)$, so $x$ divides $g_1'(x)$, implies $cx^{m_2}g_2(x)^{q^m-1}=x^{q^m-1}g_1''(x)^{q^m-1}q(x)^{m_2}x^{q^m-m_1-1}$ with $g_1''(x)=g_1'(x)/x$, which is same as $cg_2(x)^{q^m-1}=g_1''(x)^{q^m-1}q(x)^{m_2}x^{2(q^m-1)-m_1-m_2}$. Again $q^m-1>m_1$ and $q^m-1>m_2$ forces that $x|g_2(x)$, a contradiction.\\ 
  \textbf{Case 3.} $q^m-m_1-1=m_2$. Gives $cg_2(x)^{q^m-1}=g_1'(x)^{q^m-1}q(x)^{m_2}$, which is possible only  if $m_2=0$, again a contradiction.\\
  From the above discussion, it is clear that $x$ does not divide $g_1(x)$. \\ Now, let us assume $x|q(x)$ and $x\nmid g_1(x)$. Then due to irreducibility $q(x)=bx$ for some $b\in \mathbb{F}_{q^m}^*$. So by (\ref{initial}), we have \begin{align}\label{Initial1}
  x^{m_1}p(x)^{m_2}g_2(x)^{q^m-1}=dg_1(x)^{q^m-1}x^{m_2},
  \end{align} where $d=b^{m_2}$. Again three possibilities may arise, namely $m_1>m_2$, $m_2>m_1$ and $m_1=m_2$. We deal with each of them separately as follows.\\
  \textbf{Case 1.} $m_1>m_2$. Then (\ref{Initial1}) gives $x|g_1(x)$, which is not possible.\\
  \textbf{Case 2.} $m_2>m_1$. Then by (\ref{Initial1}) we get, either $x|p(x)$ or $x|g_2(x)$. But $p(x) \text{ and } q(x) $ are co-prime, therefore, $x|g_2(x).$ By (\ref{Initial1}), $p(x)^{m_2}g_2'(x)x^{q^m+m_1-m_2}=dg_1(x)$, with $g_2'(x)=g_2(x)/x$, which implies $x|g_1(x)$. A contradiction.\\
  \textbf{Case 3.} $m_1=m_2$. Here, (\ref{Initial1}) gives $p(x)^{m_2}g_2(x)^{q^m-1}=dg_1(x)^{q^m-1}$, which is possible only if $m_2=0$.

  Thus, by above discussion together with (\ref{bound})  and (\ref{support}) we get, if $(\chi_{d_1}, \chi_{d_2}, u, v)\neq (\chi_1, \chi_1, 0, 0)$ then $|\chi_{f,a,b}(\chi_{d_1}, \chi_{d_2})|\leq (n+2)q^{\frac{m}{2}+2}$. Using this and (\ref{key}), we get \begin{align}
  N_{f,n,a,b}(l_1, l_2)\geq \frac{\theta(l_1) \theta(l_2)}{q^2}(q^m-|S|-((n+2)q^{\frac{m}{2}+2})(W(l_1)W(l_2)-1))
  \end{align}
  $~~~~~~~~~~~~~~~~~~~~~\geq  \frac{\theta(l_1) \theta(l_2)}{q^2}(q^m-(n+1)-((n+2)q^{\frac{m}{2}+2})(W(l_1)W(l_2)-1))$
  
  Thus, if $q^{\frac{m}{2}-2}>(n+2)W(l_1)W(l_2)$, then $N_{f,n,a,b}(l_1, l_2)>0$ for all $f(x)\in R_n$ and prescribed $a, b\in \mathbb{F}_{q}$. The result now follows by taking $l_1=l_2=q^m-1$.
\end{proof}

For further calculation work we shall need following results. Their proofs have been omitted as they follow on ideas from  \cite{PPR5} .
\begin{lemma}\label{473} For each $M\in \mathbb{N}$, if $\omega(M)\geq 473$. Then $W(M)<M^{\frac{1}{10}}$. 
\end{lemma}
\begin{theorem}\label{Sieve} Suppose $m,n,q\in \mathbb{N}$ such that $q$ is a prime power. Also let $l|(q^m-1)$, $\{p_1,...,p_s\} $ be the collection of all primes dividing $q^m-1$ but not $l$. Suppose $\delta=1-2\sum_{i=1}^{s}\frac{1}{p_i},~ \delta>0$ and $\Delta=\frac{(2s-1)}{\delta}+2$. If $q^{\frac{m}{2}-2}>(n+2)\Delta W(l)^2$ then $(q,m)\in Q_{n}$.
\end{theorem}

\section{Computaions for $Q_2$}
By [5], for $m\leq 4$, there does not exists any primitive element $\alpha$ such that $\text{Tr}_{\mathbb{F}_{q^m}/ \mathbb{F}_q}(\alpha)=0$ and $\text{Tr}_{\mathbb{F}_{q^m}/ \mathbb{F}_q}(\alpha^{-1})=0$. The cases $m=5,$ and $6$, demand an extensive computation and seems to call for a different approach. Consequently, we defer the study of these cases on another occasion. In this paper, we consider the cases $m\geq 7$.  

 First we assume that $\omega(q^m-1) \geq 473$. Then using Lemma \ref{473} and (\ref{Main condition}),  if $q^{\frac{m}{2}-2}>4q^{\frac{m}{5}}$ i.e if $q^{\frac{3m}{10}-2}>4$ or $q^m>4^{\frac{10m}{3m-20}}$ then $(q,m)\in Q_2$. But $m\geq 7$ gives $\frac{10m}{3m-20}\leq 70$. Hence, if $q^m>4^{70}$ then $(q,m)\in Q_2$. Which is true for $\omega(q^m-1) \geq 473$.

Therefore, we can assume $\omega(q^m-1)\leq 472$. To make the further progress we use the sieving Theorem \ref{Sieve} in place of Theorem \ref{Theorem1}.  Let $31\leq \omega(q^m-1) \leq 472$. In Theorem \ref{Sieve}, let $l$ to be the product of least $31$ primes dividing $q^m-1$ i.e. $W(l)=2^{31}$. Then $s\leq 441$ and $\delta$ will be at least its value when $\{p_1, p_2, \cdots ,p_{441}\}=\{131, 137, \cdots , 3347\}$. This gives $\delta>0.0008225$ and $\Delta<1071081.2759510$, hence $4\Delta W(l)^2<1.9758 \times 10^{25}=R$(say). By Theorem \ref{Sieve} $(q,m)\in Q_2$ if $q^{\frac{m}{2}-2}>R$ that is  if $q>R^{\frac{2}{m-4}}$ or $q^m>R^{\frac{2m}{m-4}}$. But $m\geq 7$ implies $\frac{2m}{m-4}\leq \frac{14}{3}$.  Therefore, if $q^m>R^{\frac{14}{3}}$ or $q^m>1.1138\times 10^{118}$ then $(q,m)\in Q_2$. Hence, $\omega(q^m-1)\geq 62$ gives $(q,m)\in Q_2$. 

We repeat the above process of Theorem \ref{Sieve} with the values in first part of Table 1. Hence $(q,m)\in Q_2$ if $q^m>(2749163)^\frac{14}{3}$ or $q^m>1.210\times 10^{30}$ for $m=7$, and $q^m>(2749163)^4$ or $q^m>5.7122\times 10^{25}$ for $m\geq 8$. $(\because m\geq 8 \implies \frac{2m}{m-4}\leq 4)$ 

Therefore, for $m\geq 8$, it is sufficient if $\omega(q^m-1)\geq 20$. So, repeated use of Theorem \ref{Sieve} for values in second part of Table 1, provides $(q, m)\in Q_2$ if $q^m> (969830)^4 $ or $q^m>8.8468\times 10^{23}$.

\begin{center}
	$$\text{Table 1}$$
	\begin{tabular}{|m{.6cm}|m{3.2cm}|m{.7cm}|m{2cm}|m{2.2cm}|m{2.3cm} |}
		\hline
		Sr. No.    &  $a \leq \omega(q^m-1)\leq b$ & $W(l)$   &  $\delta>$   & $\Delta<$ & $4 \Delta W(l)^2$ $<$ \\
		\hline
		1.      &  $a=10,~~ b=61$ & $2^{10}$ & $0.0479926$  & $2106.4882452$ & $8835252073$  \\
		2.      &$a=7,~~ b=29$  & $2^7$ & $0.1237982$   & $349.3392467$ & $22894297$  \\
		
		3.	    & $a=6,~~ b=23$  & $2^6$ & $0.1255013$   & $264.9453729$ & $4340865$\\
		
		4.	    & $a=6,~~ b=22$  & $2^6$ & $0.1495977$   & $209.2223842$ & $3427900$\\
		
		5.	    & $a=6, ~~b=21$  & $2^6$ & $0.1749141$   & $167.7955787$ & $2749163$\\
		\hline 
		6.	    & $a=5, ~~b=19$  & $2^5$ & $0.0766343$   & $354.3225878$ & $1451306$\\
		
		7.	    & $a=5, ~~b=18$  & $2^5$ & $0.1064850$   & $236.7747170$ & $969830$\\
		\hline	
	\end{tabular}
\end{center}

 Hence $(q,m)\in Q_2$ unless $m=7$ and $q<19625$, $m=8$ and $q<985$, $m=9$ and $q<458$, $m=10$ and $q<249$, $m=11$ and $q<151$, $m=12$ and $q<99$, $m=13$ and $q<70$, $m=14$ and $q<52$, $m=15$ and $q<40$, $m=16$ and $q<32$, $m=17$ and $q<26$, $m=18$ and $q<22$, $m=19$ and $q<19$, $m=20$ and $q<16$, $m=21$ and $q<14$, $m=22$ and $q<13$, $m=23$ and $q<11$, $m=24, 25$ and $q<10$, $m=26$ and $q<9$, $m=27, 28$ and $q<8$, $29\leq m\leq 34$ and $q=2,3,4,5$. $35\leq m\leq 39$ and $q=2,3,4$.  $40\leq m \leq 50$ and $q=2, 3$. $51\leq m \leq 79$ and $q=2$. 

For each of above values we verify (\ref{Main condition}) and get a list of $494$ possible exceptions (see appendix 1). Finally, for these possible exceptions we see that Theorem \ref{Sieve} is satisfied for some choice of $l$ except the values stated in Theorem \ref{case n=2}(see appendix 2). Which proves Theorem \ref{case n=2}.

\textbf{Note:} In the case $q=4$ with $m=16, 20 ~\& ~24$ and $q=8$ with $m= 20$ equality occurs in (\ref{Main condition}), so we keep it in exception for (\ref{Main condition}) and verified using  Theorem \ref{Sieve}.

Using similar arguments, for each $n\in \mathbb{N}$, one can get a subset of $Q_n$. 
\bibliographystyle{plain}
\bibliography{Prescribedtrace}

\begin{center}
	\section*{Appendix 1.}
\end{center}

\textbf{For m=7:} 2, 4, 8, 16, 32, 64, 256, 512, 1024, 4096, 3, 9, 27, 81, 243, 729, 6561, 5, 25, 125, 625, 3125, 15625, 7, 49, 343,
2401, 11, 121, 1331, 14641, 13, 169, 2197, 19, 361, 23, 529, 29, 31, 37, 41, 1681, 43, 47, 53, 59, 3481, 61, 67,
4489, 71, 79, 6241, 83, 6889, 97, 9409, 101, 103, 107, 109, 127, 131, 17161, 139, 19321, 151, 157, 181, 191, 197,
199, 211, 223, 227, 229, 233, 239, 241, 269, 277, 281, 311, 331, 359, 367, 389, 397, 401, 409, 431, 439, 463, 491,
499, 509, 547, 571, 593, 601, 607, 613, 619, 631, 643, 661, 691, 727, 877, 919, 953, 967, 1021, 1051, 1063, 1093,
1123, 1151, 1171, 1181, 1231, 1283, 1301, 1303, 1321, 1381, 1399, 1453, 1481, 1483, 1499, 1523, 1531, 1597, 1607,
1693, 1741, 1951, 2003, 2141, 2161, 2281, 2311, 2381, 2591, 2713, 2731, 2791, 2887, 2971, 3041, 3083, 3191, 3221,
3229, 3271, 3301, 3307, 3313, 3499, 3547, 3571, 3739, 3851, 3911, 4013, 4219, 4243, 4327, 4957, 5419, 5923, 5981,
6067, 6211, 6491, 6577, 7159, 7759, 8009, 8053, 8191, 8807, 9103, 9403, 9421, 9463, 9719, 9767, 9871, 9901, 9967,
10949, 10957, 12959, 14323, 15313, 15511, 16381, 17431, 17491, 19483.  \\
\textbf{For m=8:} 2, 4, 8, 16, 32, 64, 128, 512, 3, 9, 27, 81, 243, 729, 5, 25, 125, 7, 49, 343, 11, 121, 13, 169, 17, 19, 361, 23,
529, 29, 841, 31, 961, 37, 41, 43, 47, 53, 59, 61, 67, 71, 73, 79, 83, 89, 97, 101, 103, 107, 109, 113, 127, 131,
137, 139, 149, 151, 157, 163, 167, 173, 179, 181, 191, 193, 197, 211, 223, 227, 229, 233, 239, 241, 251, 263, 269,
271, 277, 281, 283, 293, 307, 311, 313, 317, 331, 337, 347, 349, 353, 359, 367, 373, 379, 383, 389, 397, 401, 409,
419, 421, 433, 439, 443, 457, 461, 463, 467, 491, 499, 509, 521, 547, 557, 563, 571, 587, 593, 599, 601, 617, 619,
631, 647, 653, 659, 661, 683, 691, 701, 709, 727, 733, 739, 743, 757, 773, 787, 797, 809, 811, 823, 827, 829, 839,
853, 857, 859, 863, 881, 887, 911, 919, 929, 937, 941, 947, 953, 967, 971, 977, 983.\\
\textbf{For m=9:} 2, 4, 8, 16, 32, 256, 3, 9, 27, 81, 5, 25, 125, 7, 49, 11, 121, 13, 169, 19, 23, 29, 31, 37, 43, 47, 53, 61, 79, 83,
137, 139, 211, 367, 379.  \\
\textbf{For m=10:} 2, 4, 8, 16, 32, 64, 3, 9, 27, 5, 25, 125, 7, 49, 11, 13, 169, 17, 19, 23, 29, 31, 37, 41, 53, 59, 61, 89, 101, 113,
137, 139, 149.\\
\textbf{For m=11:} 2, 4, 16, 3, 9, 7, 13.  \\
\textbf{For m=12:} 2, 4, 8, 16, 32, 64, 3, 9, 27, 81, 5, 7, 49, 11, 13, 17, 19, 23, 29, 31, 37, 41, 43, 47, 89. \\
\textbf{For m=14:} 2, 4, 3, 5.  \\
\textbf{For m=15:} 2, 4, 16, 3, 9,5.\\
\textbf{For m=16:} 2, 4, 8, 3, 5. \\
\textbf{For m=18:} 2, 4, 3.  \\
\textbf{For m=20:} 2, 4, 8.  \\
\textbf{For m=22:} 2.  \\
\textbf{For m=24:} 2, 3. \\
\textbf{For m=28:} 2. \\
\textbf{For m=30:} 2. \\
\textbf{For m=36:} 2. 
\newpage

\begin{center}
	\section*{Appendix 2.}
\end{center}

\begin{center}
	
	\begin{tabular}{|m{.6cm}|m{1cm}|m{0.5cm}|m{.6cm}|m{2.6cm}|m{2.8cm}|}
		\hline
		\multicolumn{6}{|c|}{\textbf{m=7}}\\
		\hline
		Sr. No.   &  $ q $  &  $l$ & $s$   &  $\delta>$   & $\Delta<$  \\		\hline
		1 & 32 &  1 & 4 & 0.8915505547 & 9.8514897025\\2 & 64 &  3 &
		5 & 0.6457222649 & 15.9378808629\\3 & 256 &  3 & 7 & 0.3334285228 & 40.9888660087\\4 & 512 &  1 &
		6 & 0.6651810595 & 18.5368508963\\5 & 1024 &  3 & 8 & 0.6560800976 & 24.8630620765\\6 & 4096 &
		3 & 11 & 0.0036734808 & 5718.64881982\\7 & 27 &  2 & 3 & 0.8443185865 & 7.92193525032\\8 &
		81 &  2 & 5 & 0.5254270858 & 19.1289228165\\9 & 243 &  2 & 4 & 0.7881829771 & 10.8811864792\\10 &
		729 &  2 & 7 & 0.5075549201 & 27.6129917822\\11 & 6561 &  2 & 8 & 0.4766388668 & 33.4703668605\\
		12 & 125 &  2 & 4 & 0.9301005749 & 9.52606781311\\13 & 625 &  2 &
		6 & 0.1059649094 & 105.807949812\\14 & 3125 &  2 & 7 & 0.7767860571 & 18.7356247978\\15 &
		15625 &  6 & 10 & 0.5084437155 & 39.3689346910\\16 & 343 &  2 & 5 & 0.1586820930 & 58.7171747389\\
		17 & 2401 &  6 & 6 & 0.5107174129 & 23.5383296552\\18 & 121 &  6 &
		4 & 0.5534430089 & 14.6480954449\\19 & 1331 &  2 & 8 & 0.1609994433 & 95.1680239735\\20 &
		14641 &  6 & 8 & 0.4506559092 & 35.2848181769\\21 & 169 &  6 & 4 & 0.6452292307 & 12.8488575322\\
		22 & 2197 &  2 & 7 & 0.2444426756 & 55.1822030064\\23 & 361 &  6 &
		5 & 0.5869576008 & 17.3333051429\\24 & 529 &  2 & 7 & 0.0507348497 & 258.234128344\\25 & 29 &
		2 & 2 & 0.7142856915& 6.20000013363\\26 & 41 &  2 & 3 & 0.5534883544 & 11.0336137337\\27 &
		1681 &  6 & 6 & 0.2396048402& 47.9089223226\\28 & 47 &  2 & 3 & 0.8665318425 & 7.77012840672\\
		29 & 53 &  2 & 3 & 0.7771883263 & 8.43344712023\\30 & 59 &  2 & 5 & 0.8743370608& 12.2935131115\\
		31 & 3481 &  6 & 8 & 0.4736536262 & 33.6687114132\\32 & 61 &  6 &
		2 & 0.5999999999 & 7.00000000032\\33 & 67 &  6 & 3 & 0.8181666169 & 8.11122465350\\34 & 4489 &
		6 & 8 & 0.6204795746 & 26.1748489599\\
		\hline
	\end{tabular}
\end{center}

\begin{center}
	
	\begin{tabular}{|m{.6cm}|m{1cm}|m{0.5cm}|m{.6cm}|m{2.6cm}|m{2.8cm}|}
		\hline
		
		35 & 71 &  2 & 4 & 0.3120206134 & 24.4344152110\\36 &
		6241 &  6 & 8 & 0.4302431672 & 36.8640051531\\37 & 83 &  2 & 3 & 0.8822539947 & 7.66730219370\\38 & 6889 &  6 & 7 & 0.5686873450 & 24.8596611367\\39 & 97 &  2 &
		4 & 0.2847533542 & 26.5826779370\\40 & 9409 &  6 & 7 & 0.6198374698 & 22.9732399730\\41 & 101 &
		2 & 3 & 0.5718309857 & 10.7438423666\\42 & 103 &  6 & 4 & 0.7852184064 & 10.9147171573\\43 &
		107 &  2 & 4 & 0.8467870041 & 10.2665416046\\44 & 109 &  2 & 4 & 0.3085167430 & 24.6892062028\\
		45 & 127 &  6 & 4 & 0.6677386337& 12.4831436227\\46 & 131 &  2 &
		5 & 0.4303858120 & 22.9114700084\\47 & 17161 &  6 & 9 & 0.2020426137 & 86.1406656104\\48 & 139 &
		2 & 5 & 0.1767112365 & 52.9305473503\\49 & 19321 &  6 & 9 & 0.1492144906& 115.929953642\\50 &
		151 &  6 & 3 & 0.5986657769 & 10.3519055083\\51 & 157 &  2 & 4 & 0.1793272162 & 41.0347887376\\
		52 & 181 &  6 & 4 & 0.5239170445 & 15.3608938142\\53 & 191 &  2 &
		6 & 0.4686374313 & 25.4723034551\\54 & 197 &  2 & 4 & 0.6452990807 & 12.8476832038\\55 & 199 &
		2 & 6 & 0.0708057824 & 157.354543257\\56 & 211 &  6 & 4 & 0.3142791551 & 24.2731921151\\57 &
		223 &  2 & 6 & 0.2050697732 & 55.6402797130\\58 & 227 &  2 & 4 & 0.9129352855 & 9.66757524931\\
		59 & 229 &  2 & 4 & 0.2250984072 & 33.0975101345\\60 & 233 &  2 &
		4 & 0.8927131842 & 9.84126427571\\61 & 239 &  2 & 5 & 0.5276644609 & 19.0562936615\\62 & 241 &
		6 & 3 & 0.5999683101& 10.3337734932\\63 & 269 &  2 & 6 & 0.9101505170 & 14.0859130373\\64 &
		277 &  2 & 5 & 0.1773946096 & 52.7343487796\\65 & 281 &  2 & 5 & 0.2451930950 & 38.7057644785\\
		66 & 311 &  2 & 5 & 0.5252308376 & 19.1353228999\\67 & 331 &  6 &
		4 & 0.4181808978 & 18.7391672750\\68 & 359 &  2 & 6 & 0.9723990300 & 13.3122284787\\69 & 367 &
		2 & 5 & 0.2828387438 & 33.8202516257\\70 & 389 &  2 & 6 & 0.9522319837 & 13.5518068992\\71 & 397 &
		2 & 6 & 0.0667977056 & 166.676314754\\72 & 401 &  2 & 5 & 0.5250909279 & 19.1398885811\\73 &
		409 &  2 & 5 & 0.2152344921 & 43.8148592663\\
		\hline
	\end{tabular}
\end{center}

\begin{center}
	
	\begin{tabular}{|m{.6cm}|m{1cm}|m{0.5cm}|m{.6cm}|m{2.6cm}|m{2.8cm}|}
		\hline
		74 & 431 &  2 & 7 & 0.4783654121 & 29.1758778334\\
		75 & 439 &  2 & 5 & 0.3056508944 & 31.4453579701\\76 & 463 &  6 &
		4 & 0.5042985183 & 15.8806673921\\	77 & 491 &  2 & 5 & 0.3091412729 & 31.1129033442\\78 & 499 &
		2 & 5 & 0.3086755298 & 31.1568301672\\79 & 509 &  2 & 5 & 0.9143787695 & 11.8427482126\\80 & 547 &
		6 & 4 & 0.4914740431 & 16.2428681573\\81 & 571 &  6 & 4 & 0.4257713248 & 18.4407502132\\82 & 593 &
		2 & 5 & 0.9177026257 & 11.8070984515\\83 & 601 &  6 & 4 & 0.5968298848 & 13.7286352074\\84 &
		607 &  2 & 5 & 0.3134889520 & 30.7091457012\\85 & 613 &  2 & 6 & 0.1409697330 & 80.0309344810\\
		86 & 619 &  2 & 5 & 0.3133247281 & 30.7241931146\\87 & 631 &  6 &
		4 & 0.3142853814 & 24.2727508627\\88 & 643 &  2 & 5 & 0.2523820851 & 37.6602173044\\89 & 661 &  6 &
		5 & 0.3027046730 & 31.7319493281\\90 & 691 &  6 & 6 & 0.4345766850 & 27.3119883722\\

		91 & 727 &
		2 & 6 & 0.1183303978 & 94.9600525620\\92 & 877 &  2 & 6 & 0.2276201911 & 50.3261170409\\93 &
		919 &  2 & 7 & 0.0995960154 & 132.527310197\\94 & 953 &  2 & 8 & 0.4607797236 & 34.5535157679\\
		95 & 967 &  6 & 5 & 0.5713388916 & 17.7524721868\\96 & 1021 &  6 &
		5 & 0.4713606737 & 21.0936590640\\97 & 1051 &  6 & 5 & 0.2452704915 & 38.6941817730\\98 & 1063 &
		2 & 6 & 0.2926718778 & 39.5847521769\\99 & 1093 &  6 & 5 & 0.4913401654 & 20.3172486870\\100 &
		1123 &  2 & 7 & 0.0156388796 & 833.261594484\\101 & 1151 &  2 & 6 & 0.4283244566 & 27.6814660718\\
		102 & 1171 &  6 & 5 & 0.4459708445 & 22.1806914287\\103 & 1181 &  2 &
		6 & 0.5378505215 & 22.4517789968\\104 & 1231 &  6 & 6 & 0.5443279941 & 22.2084039733\\105 &
		1283 &  2 & 6 & 0.9241284167 & 13.9031076203\\106 & 1301 &  2 & 6 & 0.3306766367 & 35.2651260382\\
		107 & 1303 &  6 & 5 & 0.6497294416 & 15.8519196197\\108 & 1321 &  6 &
		5 & 0.3490947277 & 27.7809679830\\109 & 1381 &  6 & 5 & 0.4839161508 & 20.5982633252\\110 &
		1399 &  2 & 6 & 0.2400354471 & 47.8265649199\\111 & 1453 &  2 & 6 & 0.1494258883 & 75.6150885335\\
		112 & 1481 &  2 & 6 & 0.5428813498 & 22.2622543625\\
		
		\hline
	\end{tabular}
\end{center}

\begin{center}
	
	\begin{tabular}{|m{.6cm}|m{1cm}|m{0.5cm}|m{.6cm}|m{2.6cm}|m{2.8cm}|}
		\hline
		113 & 1483 &  2 &
		7 & 0.0175562934 & 742.475204429\\114 & 1499 &  2 & 6 & 0.6087996314 & 20.0683420812\\115 &
		1523 &  2 & 6 & 0.9588049074 & 13.4726154557\\
		116 & 1531 &  6 & 6 & 0.4102176092 & 28.8150360946\\
		117 & 1597 &  6 & 5 & 0.6063304365 & 16.8433914220\\118 & 1607 &  2 &
		6 & 0.7161038245& 17.3609010635\\119 & 1693 &  2 & 7 & 0.2380189853 & 56.6174918756\\120 &
		1741 &  6 & 6 & 0.4563537984 & 26.1041052716\\121 & 1951 &  6 & 7 & 0.3716254882 & 36.9814542067\\
		122 & 2003 &  2 & 7 & 0.3768219666 & 36.4990503533\\123 & 2141 &  2 &
		7 & 0.5104339094 & 27.4685273810\\124 & 2161 &  6 & 7 & 0.5155635456 & 27.2151264589\\125 & 2281 &
		6 & 6 & 0.4751049549 & 25.1527789517\\126 & 2311 &  6 & 6 & 0.0635019652 & 175.222985454\\127 &
		2381 &  2 & 8 & 0.1487999594 & 102.806479076\\128 & 2591 &  2 & 7 & 0.2053401141 & 65.3095975976\\
		129 & 2713 &  2 & 7 & 0.2001186347 & 66.9614665471\\130 & 2731 &  6 &
		7 & 0.1351215775 & 98.2096523175\\131 & 2791 &  6 & 7 & 0.4597776097 & 30.2745390926\\132 &
		2887 &  2 & 7 & 0.0457255113 & 286.305185939\\133 & 2971 &  6 &
		6 & 0.3663987435 & 32.0219370123\\134 & 3041 &  2 & 7 & 0.4196313309 & 32.9795743074\\135 &
		3083 &  2 & 7 & 0.8518534046 & 17.2608417466\\136 & 3191 &  2 & 7 & 0.3330309743 & 41.0354081178\\ 
		137 & 3221 &  2 & 7 & 0.2241558119 & 59.9953733283\\138 & 3229 &  2 &
		7 & 0.2781932675 & 48.7301028319\\139 & 3271 &  6 & 6 & 0.5126290898 & 23.4580097349\\140 &
		3301 &  6 & 7 & 0.3469499305 & 39.4693834860\\141 & 3307 &  2 & 7 & 0.1383387900 & 95.9721967522\\
		142 & 3313 &  2 & 7 & 0.1679193424 & 79.4181211729\\143 & 3499 &  2 &
		9 & 0.0119063509 & 1429.80941247\\144 & 3547 &  2 & 7 & 0.2574629280 & 52.4927062684\\145 &
		3571 &  6 & 6 & 0.1965247699 & 57.9725880870\\146 & 3739 &  6 & 6 & 0.6450851929 & 19.0520113005\\
		147 & 3851 &  2 & 7 & 0.0620442258 & 211.527958983\\148 & 3911 &  2 &
		7 & 0.2769475060 & 48.9403035387\\149 & 4013 &  2 & 7 & 0.8129600979 & 17.9909447376\\150 &
		4219 &  2 & 8 & 0.1394521022 & 109.563814080\\151 & 4243 &  2 &
		8 & 0.0205124685 & 733.262547676\\
		\hline
	\end{tabular}
\end{center}

\begin{center}
	\begin{tabular}{|m{.6cm}|m{1cm}|m{0.5cm}|m{.6cm}|m{2.6cm}|m{2.8cm}|}
		\hline
		152 & 4327 &  2 & 8 & 0.0268222607 & 561.236975841\\153 &
		4957 &  6 & 7 & 0.6599172480 & 21.6994396474\\154 & 5419 &  6 & 7 & 0.5879914646 & 24.1091644709\\
		155 & 5923 &  6 & 7 & 0.5744529082 & 24.6302274977\\156 & 5981 &  2 &
		8 & 0.2435838274 & 63.5804429976\\157 & 6067 &  2 & 8 & 0.2508963975& 61.7856332307\\158 & 6211 &
		6 & 7 & 0.4850483309 & 28.8014529098\\159 & 6491 &  2 & 8 & 0.2406360147 & 64.3348089174\\160 &
		6577 &  2 & 8 & 0.1750366819 & 87.6963228157\\161 & 7159 &  2 & 8 & 0.2139241299 & 72.1183172098\\
		162 & 7759 &  2 & 9 & 0.2220720881 & 78.5517185881\\163 & 8009 &  2 &
		9 & 0.3122606079 & 56.4417053129\\164 & 8053 &  6 & 9 & 0.6622828307 & 27.6687916572\\165 &
		8191 &  6 & 7 & 0.1037179298 & 127.339948604\\166 & 8807 &  2 & 8 & 0.4094012243 & 38.6388743060\\
		167 & 9103 &  2 & 8 & 0.2251270564 & 68.6290415515\\168 & 9403 &  2 &
		8 & 0.2612195377 & 59.4229635688\\169 & 9421 &  6 & 7 & 0.4630157204 & 30.0768004726\\170 &
		9463 &  2 & 8 & 0.1753820443 & 87.5275695620\\171 & 9719 &  2 & 8 & 0.9276007580 & 18.1707500444\\
		172 & 9767 &  2 & 8 & 0.7880692489 & 21.0338603111\\173 & 9871 &  6 &
		7 & 0.2710577749& 49.9602549701\\174 & 9901 &  6 & 7 & 0.3422488091 & 39.9840620444\\175 &
		9967 &  2 & 8 & 0.0689507325 & 219.546637155\\
		176 & 10949 &  2 & 10 & 0.4245538380 & 46.7528635858\\177 & 10957 &  6 &
		8 & 0.6547596918& 24.9091683356\\178 & 12959 &  2 & 9 & 0.5016182530 & 35.8903137936\\179 &
		14323 &  6 & 8 & 0.4181086282 & 37.8758441839\\180 & 15313 &  2 &
		9 & 0.046358850 & 368.704519260\\181 & 15511 &  6 & 8 & 0.2887683586 & 53.9447493156\\182 &
		16381 &  6 & 8 & 0.0242915320 & 619.499133949\\183 & 17431 &  6 &
		9 & 0.2272887003 & 76.7947433194\\184 & 17491 &  6 & 9 & 0.3676629404 & 48.2380026110\\185 &
		19483 &  2 & 9 & 0.0840160090 & 204.342389188\\
		
		\hline
	\end{tabular}
\end{center}

\begin{center}
	
	\begin{tabular}{|m{.6cm}|m{1cm}|m{0.5cm}|m{.6cm}|m{2.6cm}|m{2.8cm}|}
		\hline
		\multicolumn{6}{|c|}{\textbf{m=8}}\\
		\hline
		Sr. No.   &  $ q $  &  $l$ & $s$   &  $\delta>$   & $\Delta<$  \\		\hline
		1 & 32 &  3 & 6 & 0.1872057176 & 60.75888909556\\2 & 64 &  15 &
		7 & 0.4031213203 & 34.24835637180\\3 & 128 &  3 & 7 & 0.3334285228 & 40.98886600865\\4 & 256 &  3 &
		6 & 0.4714198996 & 25.33376255281\\5 & 512 &  15 & 10 & 0.2247600824 & 86.53458367679\\6 & 81 &  2 &
		5 & 0.4232096661 & 23.26605491293\\7 & 243 &  2 & 6 & 0.3349209179 & 34.84357413644\\8 & 729 &  10 &
		10 & 0.3292579234 & 59.70552095249\\9 & 125 &  6 & 6 & 0.4862007398 & 24.62439996228\\10 & 49 &
		6 & 4 & 0.4806758665 & 16.56282806681\\11 & 343 &  6 & 9 & 0.2390141275 & 73.12550281118\\12 & 121 &
		6 & 5 & 0.4492925519 & 22.03149164234\\13 & 169 &  6 & 5 & 0.1964986043 & 47.80185202880\\14 &
		361 &  6 & 6 & 0.4706451953 & 25.37217103169\\15 & 529 &  6 & 7 & 0.2622404271 & 51.57282956953\\
		16 & 841 &  6 & 7 & 0.1918025335 & 69.77804098258\\17 & 961 &  6 &
		7 & 0.2677912784 & 50.54527031163\\18 & 53 &  6 & 5 & 0.3213807316 & 30.00416799836\\19 & 59 &  6 &
		6 & 0.4055381905 & 29.12444908078\\20 & 61 &  6 & 4 & 0.5344088910 & 15.09858446835\\21 & 67 &  6 &
		6 & 0.2937599500 & 39.44554013625\\22 & 71 &  6 & 4 & 0.3134922209 & 24.32910271236\\23 & 73 &  6 &
		5 & 0.3433191634 & 28.21467415274\\24 & 79 &  6 & 6 & 0.3902904463 & 30.18413850894\\25 & 89 &  6 &
		6 & 0.2848157160& 40.62146426212\\26 & 97 &  6 & 5 & 0.3035660971 & 31.64757949958\\27 & 101 &  6 &
		5 & 0.4594855299& 21.58712388658\\28 & 103 &  6 & 5 & 0.3266217376 & 29.55481023604\\29 & 107 &
		6 & 5 & 0.5529262524 & 18.27703506580\\30 & 109 &  6 & 6 & 0.2111776470 & 54.08884629189\\31 &
		113 &  6 & 5 & 0.2074563611 & 45.38261767584\\32 & 127 &  6 & 6 & 0.1807643788 & 62.85269713586\\
		33 & 131 &  6 & 6 & 0.2464019414 & 46.64250539872\\34 & 137 &  6 &
		7 & 0.3434750023 & 39.84846032837\\35 & 139 &  6 & 6 & 0.2238740722 & 51.13476531064\\36 & 149 &
		6 & 6 & 0.4248936099 & 27.88883368115\\
		
		\hline
	\end{tabular}
\end{center}

\begin{center}
	
	\begin{tabular}{|m{.6cm}|m{1cm}|m{0.5cm}|m{.6cm}|m{2.6cm}|m{2.8cm}|}
		\hline
		37 & 151 &  6 & 6 & 0.2209629969 & 51.78209090649\\
		38 &
		157 &  6 & 7 & 0.2165249554 & 62.0392688024\\39 & 163 &  6 & 5 & 0.5471355910& 18.4493046111\\
		40 & 167 &  6 & 6 & 0.2406915271 & 47.7016502774\\41 & 173 &  6 &
		7 & 0.4057408925 & 34.0401523226\\42 & 179 &  6 & 6 & 0.4012069133 & 29.4172743143\\43 & 181 &
		6 & 6 & 0.1591853526 & 71.1018351472\\44 & 191 &  6 & 7 & 0.2052896009 & 65.3251754578\\45 &
		193 &  6 & 5 & 0.5581757824 & 18.1239528537\\46 & 197 &  6 & 6 & 0.1045547472 & 107.208039735\\47 &
		211 &  6 & 7 & 0.2403992243 & 56.0767135718\\48 & 223 &  6 & 7 & 0.1402470073 & 94.6935999903\\
		49 & 227 &  6 & 7 & 0.4538411986 & 30.6443805416\\50 & 229 &  6 &
		8 & 0.1078964707 & 141.022156135\\51 & 233 &  6 & 8 & 0.2989874283 & 52.1693334864\\52 & 239 &
		6 & 7 & 0.0374699277 & 348.944891005\\53 & 241 &  6 & 7 & 0.3446763887 & 39.7165376721\\54 &
		251 &  6 & 5 & 0.1782900306 & 52.4795471158\\55 & 263 &  6 & 7 & 0.2842775409 & 47.7299579721\\
		56 & 269 &  6 & 6 & 0.5440800215 & 22.2176142563\\57 & 271 &  6 &
		6 & 0.4727668348& 25.2672835513\\58 & 277 &  6 & 6 & 0.4948790097 & 24.2276552122\\59 & 281 &
		30 & 6 & 0.3995807546 & 29.5288533626\\60 & 283 &  6 & 6 & 0.5233621328 & 23.0179516432\\61 &
		293 &  6 & 7 & 0.1493475784 & 89.0452680607\\62 & 307 &  30 & 7 & 0.1919345878 & 69.7314086351\\
		63 & 311 &  6 & 7 & 0.3606818465 & 38.0428453068\\64 & 313 &  6 & 6 & 0.3929944553 & 29.9902167847\\
		65 & 317 &  6 & 6 & 0.3805142188 & 30.9082495585\\66 & 331 &  6 &
		8 & 0.1576333752 & 97.1575132839\\67 & 337 &  6 & 6 & 0.1044388557 & 107.324784769\\68 & 347 &
		6 & 6 & 0.5163357360 & 23.3039680027\\69 & 349 &  6 & 6 & 0.1276402934 & 88.1796827816\\70 &
		353 &  6 & 6 & 0.2639079410 & 43.6812012402\\71 & 359 &  6 & 7 & 0.3167046688 & 43.0477055705\\
		72 & 367 &  6 & 6 & 0.4800483197 & 24.9143599672\\73 & 373 &  6 &
		7 & 0.2358267874 & 57.1252049975\\74 & 379 &  6 & 6 & 0.2082414656 & 54.8232932204\\75 & 383 &
		6 & 7 & 0.4052511510 & 34.0788724907\\	
		\hline
	\end{tabular}
\end{center}

\begin{center}
	
	\begin{tabular}{|m{.6cm}|m{1cm}|m{0.5cm}|m{.6cm}|m{2.6cm}|m{2.8cm}|}
		\hline
		
		76 & 389 &  6 & 8 & 0.2380009225 & 65.0249657815\\77 & 397 &
		6 & 6 & 0.4066612483 & 29.0495407313\\78 & 401 &  6 & 8 & 0.4137658068 & 38.2523914504\\79 & 409 &
		6 & 6 & 0.4330173581& 27.4031386794\\80 & 419 &  30 & 6 & 0.3774897437 & 31.1398645478\\81 &
		421 &  30 & 8 & 0.3931075085 & 40.1575006199\\82 & 433 &  6 & 7 & 0.1095437717 & 120.674022146\\
		83 & 439 &  6 & 7 & 0.3735860118 & 36.7978767595\\84 & 443 &  6 &
		7 & 0.2616809046 & 51.6788254997\\85 & 457 &  6 & 7 & 0.3454052974 & 39.6369444740\\86 & 461 &
		30 & 7 & 0.3259369732 & 41.8850117232\\87 & 463 &  30 & 7 & 0.1713902457 & 77.8502909061\\88 &
		467 &  6 & 7 & 0.2918612855 & 46.5417074576\\89 & 491 &  6 & 7 & 0.1319631598 & 100.512342527\\90 &
		499 &  6 & 8 & 0.3710824050 & 42.4222884053\\91 & 509 &  6 & 7 & 0.4548338866 & 30.5818633591\\
		92 & 521 &  6 & 7 & 0.3320767696 & 41.1475742599\\93 & 547 &  6 &
		8 & 0.0792943677 & 191.168542789\\94 & 557 &  6 & 7 & 0.3759720450 & 36.5770388319\\95 & 563 &
		6 & 9 & 0.3473154146 & 50.9468629502\\96 & 571 &  6 & 9 & 0.0832918468 & 206.101609618\\97 &
		587 &  6 & 7 & 0.1896844363 & 70.5348795789\\98 & 593 &  6 & 8 & 0.0888507896 & 170.822360123\\
		99 & 599 &  6 & 8 & 0.1967209780 & 78.2501292236\\100 & 601 &  6 &
		7 & 0.2091376164 & 64.1600275479\\101 & 617 &  6 & 8 & 0.1103965419 & 137.873821201\\102 & 619 &
		6 & 7 & 0.3620072268 & 37.9108852927\\103 & 631 &  6 & 7 & 0.1649222429 & 80.8250254445\\104 &
		647 &  6 & 7 & 0.2989531696 & 45.4850716386\\105 & 653 &  6 & 7 & 0.5198078501 & 27.0092413877\\
		106 & 659 &  30 & 8 & 0.2774521892 & 56.0633686880\\107 & 661 &  6 &
		7 & 0.2941420522 & 46.1963326884\\108 & 683 &  6 & 7 & 0.1995791693 & 67.1370583620\\109 & 691 &
		6 & 7 & 0.4894732779 & 28.5591618309\\110 & 701 &  6 & 7 & 0.0087511261 & 1487.52310204\\111 &
		709 &  6 & 7 & 0.4827427717 & 28.9294555238\\112 & 727 &  30 & 7 & 0.2564922470 & 52.6837931671\\
		113 & 733 &  6 & 7 & 0.2897863968 & 46.8606288745\\114 & 739 &  6 &
		7 & 0.3792397330 & 36.2791086112\\
		\hline
	\end{tabular}
\end{center}

\begin{center}
	
	\begin{tabular}{|m{.6cm}|m{1cm}|m{0.5cm}|m{.6cm}|m{2.6cm}|m{2.8cm}|}
		\hline
		115 & 743 &  6 & 9 & 0.1466957822 & 117.8860857583\\116 & 757 &
		6 & 8 & 0.1462365540 & 104.5735329923\\117 & 773 &  6 & 8 & 0.4073178011 & 38.82628149128\\118 &
		787 &  6 & 8 & 0.5489889240 & 29.32295560762\\119 & 797 &  6 & 8 & 0.0791949641 & 191.4059825400\\
		120 & 809 &  6 & 8 & 0.5290676509 & 30.35176176875\\121 & 811 &  6 &
		7 & 0.0394520642 & 331.5138095121\\122 & 823 &  6 & 8 & 0.4939009033 & 32.37046479958\\123 &
		827 &  6 & 8 & 0.0391455202 & 385.1856087699\\124 & 829 &  6 & 7 & 0.2535540288 & 53.27112378949\\
		125 & 839 &  6 & 7 & 0.2905223004 & 46.74699525357\\126 & 853 &  30 &
		9 & 0.3952108075 & 45.01501799704\\127 & 857 &  6 & 7 & 0.1865522381 & 71.68557510454\\128 & 859 &
		6 & 8 & 0.0848357711 & 178.8122078036\\129 & 863 &  6 & 8 & 0.3031669975 & 51.47768101438\\130 &
		881 &  6 & 7 & 0.1308872871 & 101.3220982703\\131 & 887 &  6 & 7 & 0.4701769660 & 29.64916391032\\
		132 & 911 &  30 & 9 & 0.3279388663 & 53.83893018476\\133 & 919 &  6 &
		7 & 0.3038412701 & 44.78549781670\\134 & 929 &  6 & 7 & 0.4176858634 & 33.12386876449\\135 & 937 &
		30 & 7 & 0.3923004102 & 35.13787000919\\136 & 941 &  6 & 7 & 0.3907873557 & 35.26617355448\\137 &
		947 &  6 & 8 & 0.3432078894 & 45.70528901861\\138 & 953 &  6 & 7 & 0.1583008901 & 84.12209034640\\
		139 & 967 &  30 & 8 & 0.1512678329 & 101.1618621765\\140 & 971 &  6 &
		8 & 0.4229069740 & 37.46879318756\\141 & 977 &  6 & 8 & 0.3979295188 & 39.69511757173\\142 & 983 &
		6 & 7 & 0.3929676394 & 35.08160442065\\

		\hline
	\end{tabular}
\end{center}

\begin{center}
	
	\begin{tabular}{|m{.6cm}|m{1cm}|m{0.5cm}|m{.6cm}|m{2.6cm}|m{2.8cm}|}
		
		\hline
		\multicolumn{6}{|c|}{\textbf{m=9}}\\
		\hline
		Sr. No.   &  $ q $  &  $l$ & $s$   &  $\delta>$   & $\Delta<$  \\		\hline
		1 & 8 &  1 & 3 & 0.6868808395 & 9.27928297361\\2 & 32 &  1 &
		6 & 0.6058719233 & 20.1556523351\\3 & 256 &  15 & 10 & 0.2247600824 & 86.5345836768\\4 & 27 &  2 &
		5 & 0.8203006416 & 12.9715871750\\5 & 81 &  10 & 7 & 0.3710793117 & 37.0329419816\\6 & 25 &  6 &
		5 & 0.5417068103 & 18.6141533179\\7 & 125 &  2 & 7 & 0.8015341751 & 18.2188967138\\8 & 49 &  2 &
		6 & 0.1256059766 & 89.5754505734\\9 & 121 &  6 & 6 & 0.1549639848 & 72.9842355352\\10 & 13 &  2 &
		3 & 0.3005452055 & 18.6364324132\\11 & 169 &  6 & 7 & 0.5539810296 & 25.4665075241\\12 & 19 &  2 &
		4 & 0.3136945189 & 24.3147029272\\13 & 23 &  2 & 5 & 0.4018876622 & 24.3943177322\\14 & 29 &  2 &
		5 & 0.5304017410 & 18.9682700922\\15 & 31 &  6 & 4 & 0.5933832129 & 13.7967610941\\16 & 37 &  6 &
		5 & 0.6412679962 & 16.0346938448\\17 & 43 &  6 & 6 & 0.5841034477 & 20.8322805539\\18 & 47 &  2 &
		5 & 0.7209393775 & 14.4837126122\\19 & 53 &  2 & 7 & 0.4825438637 & 28.9405560317\\20 & 61 &  6 &
		5 & 0.3202721293 & 30.1011027042\\21 & 79 &  6 & 5 & 0.5088901393 & 19.6855460618\\22 & 83 &  2 &
		6 & 0.7864070130 & 15.9876677309\\23 & 137 &  2 & 7 & 0.4098522021 & 33.7187511280\\24 & 139 &
		6 & 6 & 0.6499143708 & 18.9253066148\\25 & 211 &  6 & 7 & 0.0418582251 & 312.572174492\\26 &
		367 &  6 & 8 & 0.6259728662 & 25.9626999984\\27 & 379 &  6 & 8 & 0.6023697252 & 26.9016498843\\

		\hline
	\end{tabular}
\end{center}

\begin{center}
	
	\begin{tabular}{|m{.6cm}|m{1cm}|m{0.5cm}|m{.6cm}|m{2.6cm}|m{2.8cm}|}
		\hline
		\multicolumn{6}{|c|}{\textbf{m=10}}\\
		\hline
		Sr. No.   &  $ q $  &  $l$ & $s$   & $\delta>$   & $\Delta<$  \\		\hline
		
		1 & 8 &  3 & 5 & 0.4486640742 & 22.0595512691\\2 & 16 &  3 &
		6 & 0.1872057176 & 60.7588890956\\3 & 32 &  1 & 7 & 0.0740989085 & 177.441180654\\4 & 64 &  15 &
		9 & 0.2117365428 & 82.2884555226\\5 & 9 &  2 & 4 & 0.3837014528 & 20.2433502618\\6 & 27 &  2 &
		7 & 0.2734997902 & 49.5320291466\\7 & 25 &  6 & 6 & 0.5833290740 & 20.8572805452\\8 & 125 &  6 &
		9 & 0.3906958109 & 45.5121122997\\9 & 49 &  6 & 6 & 0.3993817568 & 29.5425700126\\10 & 11 &  6 &
		3 & 0.5992300546 & 10.3440407593\\11 & 13 &  6 & 4 & 0.5315733620 & 15.1684551944\\12 & 169 &  6 &
		9 & 0.0027277975 & 6234.13410652\\13 & 17 &  2 & 5 & 0.1035216197 & 88.9383615033\\14 & 19 &
		6 & 5 & 0.4018529013 & 24.3962548709\\15 & 23 &  2 & 6 & 0.0287330255 & 384.834727787\\16 &
		29 &  6 & 6 & 0.0629611420 & 176.710935087\\17 & 31 &  6 & 5 & 0.3691944084 & 26.3774006183\\18 &
		37 &  6 & 5 & 0.6636688020 & 15.5609809768\\19 & 41 &  6 & 6 & 0.0991906950 & 112.897498971\\20 &
		53 &  6 & 6 & 0.6471962543 & 18.9963900846\\21 & 59 &  6 & 7 & 0.2761402170 & 49.0775323510\\22 &
		61 &  6 & 6 & 0.3383038363 & 34.5151500418\\23 & 89 &  6 & 7 & 0.3326460941 & 41.0805731024\\24 &
		101 &  6 & 7 & 0.2304968711 & 58.3998979067\\25 & 113 &  6 & 8 & 0.3376043565 & 46.4307062605\\
		26 & 137 &  6 & 8 & 0.5629731164 & 28.6442562897\\27 & 139 &  30 &
		8 & 0.3194991731 & 48.9484783125\\28 & 149 &  6 & 8 & 0.3197656340 & 48.9093561186\\

		\hline
	\end{tabular}
\end{center}

\begin{center}
	
	\begin{tabular}{|m{.6cm}|m{1cm}|m{0.5cm}|m{.6cm}|m{2.6cm}|m{2.8cm}|}
		\hline
		\multicolumn{6}{|c|}{\textbf{m=11}}\\
		\hline
		Sr. No.   &  $q $  &  $l$ & $s$   & $\delta>$   & $\Delta<$  \\		\hline
		
		1 & 4 &  3 & 3 & 0.8876433104 & 7.63289323658\\2 & 16 &  3 &
		6 & 0.4816590055 & 24.8377334861\\3 & 9 &  2 & 4 & 0.8796476677 & 9.95773155158\\4 & 7 &  2 &
		3 & 0.3315455741 & 17.0808829601\\5 & 13 &  2 & 5 & 0.2391643945 & 39.6310195135\\

		\hline
	\end{tabular}
\end{center}

\begin{center}
	
	\begin{tabular}{|m{.6cm}|m{1cm}|m{0.5cm}|m{.6cm}|m{2.6cm}|m{2.8cm}|}
		\hline
		\multicolumn{6}{|c|}{\textbf{m=12}}\\
		\hline
		Sr. No.   &  $q $  &  $l$ & $s$   & $\delta>$   & $\Delta<$  \\		\hline
		
		1 & 8 &  15 & 6 & 0.3553764643 & 32.9530908855\\2 & 16 &  15 &
		7 & 0.4031213203 & 34.2483563718\\3 & 32 &  15 & 9 & 0.2117365428 & 82.2884555226\\4 & 64 &  15 &
		10 & 0.2247600824 & 86.5345836768\\5 & 9 &  2 & 6 & 0.0839532181 & 133.025352703\\6 & 27 &  10 &
		7 & 0.3710793117 & 37.0329419816\\7 & 81 &  10 & 10 & 0.3292579234 & 59.7055209525\\8 & 7 &  6 &
		5 & 0.2833293365 & 33.7651539659\\9 & 49 &  6 & 9 & 0.2390141275 & 73.1255028112\\10 & 11 &  30 &
		6 & 0.3665449529 & 32.0099617018\\11 & 13 &  6 & 6 & 0.1510424768 & 74.8271955803\\12 & 17 &  6 &
		6 & 0.0849353562 & 131.510259107\\13 & 19 &  6 & 6 & 0.1310410249 & 85.9431773675\\14 & 23 &
		30 & 7 & 0.2612500840 & 51.7607495390\\15 & 29 &  6 & 9 & 0.0252274308 & 675.869649089\\16 & 31 &
		30 & 6 & 0.3950778845 & 29.8426113697\\17 & 37 &  30 & 8 & 0.2996854761 & 52.0524756523\\18 & 41 &
		6 & 8 & 0.0673048338 & 224.866607782\\19 & 43 &  30 & 7 & 0.3070086779 & 44.3440799377\\20 &
		47 &  30 & 9 & 0.2474313715 & 70.7059199300\\21 & 89 &  30 & 9 & 0.2368116110 & 73.7870205977\\

		\hline
	\end{tabular}
\end{center}

\begin{center}
	
	\begin{tabular}{|m{.6cm}|m{1cm}|m{0.5cm}|m{.6cm}|m{2.6cm}|m{2.8cm}|}	\hline
		\multicolumn{6}{|c|}{\textbf{m=14}}\\
		\hline
		Sr. No.   &  $q $  &  $l$ & $s$   & $\delta>$   & $\Delta<$  \\		\hline
		
		1 & 4 &  3 & 5 & 0.4510757083 & 21.9523047555\\2 & 3 &  2 &
		2 & 0.9945138667 & 5.01654919096\\3 & 5 &  2 & 4 & 0.2598110717 & 28.9426547206\\

		\hline
	\end{tabular}
\end{center}

\begin{center}
	
	\begin{tabular}{|m{.6cm}|m{1cm}|m{0.5cm}|m{.6cm}|m{2.6cm}|m{2.8cm}|}
		\hline
		\multicolumn{6}{|c|}{\textbf{m=15}}\\
		\hline
		Sr. No.   &  $q $  &  $l$ & $s$   & $\delta>$   & $\Delta<$  \\		\hline
		1 & 2 &  1 & 3 & 0.6365245521 & 9.85515654217\\2 & 4 &  3 &
		5 & 0.4486640742 & 22.0595512691\\3 & 16 &  15 & 9 & 0.2117365428 & 82.2884555226\\4 & 3 &  2 &
		3 & 0.6638971640 & 9.53128687857\\5 & 9 &  2 & 7 & 0.2734997902 & 49.5320291466\\6 & 5 &  2 &
		5 & 0.7132981862 & 14.6174441121\\

		\hline
	\end{tabular}
\end{center}

\begin{center}
	
	\begin{tabular}{|m{.6cm}|m{1cm}|m{0.5cm}|m{.6cm}|m{2.6cm}|m{2.8cm}|}
		\hline
		\multicolumn{6}{|c|}{\textbf{m=16}}\\
		\hline
		Sr. No.   &  $q $  &  $l$ & $s$   & $\delta>$   & $\Delta<$  \\		\hline
		1 & 4 &  3 & 4 & 0.4745403228 & 16.7511173704\\2 & 8 &  3 &
		8 & 0.0031213203 & 4807.65859167\\3 & 3 &  2 & 4 & 0.4232097591 & 18.5402613007\\4 & 5 &  2 &
		5 & 0.0552762647 & 164.818526912\\

		\hline
	\end{tabular}
\end{center}
\begin{center}
	
	\begin{tabular}{|m{.6cm}|m{1cm}|m{0.5cm}|m{.6cm}|m{2.6cm}|m{2.8cm}|}
		\hline
		\multicolumn{6}{|c|}{\textbf{m=18}}\\
		\hline
		Sr. No.   &  $q $  &  $l$ & $s$   & $\delta>$   & $\Delta<$  \\		\hline
		1 & 4 &  15 & 6 & 0.3553764643 & 32.9530908855\\2 & 3 &  2 &
		5 & 0.3984803405 & 24.5858068362\\

		\hline
	\end{tabular}
\end{center}

\begin{center}
	
	\begin{tabular}{|m{.6cm}|m{1cm}|m{0.5cm}|m{.6cm}|m{2.6cm}|m{2.8cm}|}
		\hline
		\multicolumn{6}{|c|}{\textbf{m=20}}\\
		\hline
		Sr. No.   &  $q $  &  $l$ & $s$   & $\delta>$   & $\Delta<$  \\		\hline
		1 & 4 &  3 & 6 & 0.1872057176 & 60.7588890956\\2 & 8 &  15 &
		9 & 0.2117365429 & 82.2884555226\\

		\hline
	\end{tabular}
\end{center}

\begin{center}
	
	\begin{tabular}{|m{.6cm}|m{1cm}|m{0.5cm}|m{.6cm}|m{2.6cm}|m{2.8cm}|}
		\hline
		\multicolumn{6}{|c|}{\textbf{m=22}}\\
		\hline
		Sr. No.   &  $q $  &  $l$ & $s$   & $\delta>$   & $\Delta<$  \\		\hline
		1 & 2 &  1 & 4 & 0.2209766437 & 33.6775559615\\
		
		\hline
	\end{tabular}
\end{center}

\begin{center}
	
	\begin{tabular}{|m{.6cm}|m{1cm}|m{0.5cm}|m{.6cm}|m{2.6cm}|m{2.8cm}|}
		\hline
		\multicolumn{6}{|c|}{\textbf{m=24}}\\
		\hline
		Sr. No.   &  $q $  &  $l$ & $s$   & $\delta>$   & $\Delta<$  \\		\hline
		1 & 4 &  3 & 8 & 0.0031213203 & 4807.65859167\\2 & 3 &  2 &
		6 & 0.0839532181 & 133.025352703\\
		
		\hline
	\end{tabular}
\end{center}

\begin{center}
	\begin{tabular}{|m{.6cm}|m{1cm}|m{0.5cm}|m{.6cm}|m{2.6cm}|m{2.8cm}|}
		\hline
		\multicolumn{6}{|c|}{\textbf{m=28}}\\
		\hline
		Sr. No.   &  $q $  &  $l$ & $s$   & $\delta>$   & $\Delta<$  \\		\hline
		1 & 2 &  3 & 5 & 0.4510757083 & 21.9523047555\\
		
		\hline
	\end{tabular}
\end{center}

\begin{center}
	
	\begin{tabular}{|m{.6cm}|m{1cm}|m{0.5cm}|m{.6cm}|m{2.6cm}|m{2.8cm}|}
		\hline
		\multicolumn{6}{|c|}{\textbf{m=30}}\\
		\hline
		Sr. No.   &  $q $  &  $l$ & $s$   & $\delta>$   & $\Delta<$  \\		\hline
		1 & 2 &  3 & 5 & 0.4486640742 & 22.0595512691\\
		
		\hline
	\end{tabular}
\end{center}

\begin{center}
	
	\begin{tabular}{|m{.6cm}|m{1cm}|m{0.5cm}|m{.6cm}|m{2.6cm}|m{2.8cm}|}
		\hline
		\multicolumn{6}{|c|}{\textbf{m=36}}\\
		\hline
		Sr. No.   &  $q $  &  $l$ & $s$   & $\delta>$   & $\Delta<$  \\		\hline
		1 & 2 &  15 & 6 & 0.3553764643 & 32.9530908855\\
		
		\hline
	\end{tabular}
\end{center}

\end{document}